\documentclass[12pt, titlepage]{article}
\usepackage{eurosym}
\usepackage{amsmath,amssymb,amsthm,amsfonts}
\usepackage{epstopdf}
\usepackage{subcaption}
\usepackage{appendix}
\usepackage{amsmath}
\usepackage{amsfonts}
\usepackage{amssymb}
\usepackage{multirow}
\usepackage{rotating}
\usepackage{xcolor}
\usepackage{lscape}
\usepackage{verbatim}
\usepackage{graphicx}
\usepackage{cancel}
\usepackage[normalem]{ulem}
\usepackage{url}
\numberwithin{equation}{section} 

\newtheorem{theorem}{Theorem}[section]
\newtheorem{corollary}{Corollary}[section]

\newtheorem{proposition}{Proposition}[section]
\newtheorem{definition}{Definition}[section]
\newtheorem{remark}{Remark}
\newtheorem{example}{Example}

\def\XX{\boldsymbol{X}}
\def\xx{\boldsymbol{x}}
\def\EE{\boldsymbol{E}}

\def\gg{\boldsymbol{g}}

\def\ee{\boldsymbol{e}}

\def\GGG{\mathcal{G}}
\def\EEE{\mathcal{E}_d}
\def\SSS{\mathcal{S}_d}

\def\YY{\boldsymbol{Y}}

\def\pp{\boldsymbol{p}}
\def\dd{\boldsymbol{d}}
\def\fff{\boldsymbol{f}^{\chi}}
\def\mmu{\boldsymbol{\mu}}

\def\RR{\mathbb{R}}

\def\NNN{\boldsymbol{N}}

\def\ff{\boldsymbol{f}}

\def\DDD{\mathcal{D}_d}

\DeclareMathOperator\expval{E}

\DeclareMathOperator{\rank}{rank}

\begin{document}

\author{ R.FONTANA\footnote{%
Corresponding author: Roberto Fontana
Department of Mathematical Sciences G. Lagrange, Politecnico di
Torino. Email: roberto.fontana@polito.it}   \\ \textit{\ %EndAName
Department of Mathematical Sciences G. Lagrange,} \\ {Politecnico di
Torino.}
\\ P. SEMERARO\\ \textit{\ %EndAName
Department of Mathematical Sciences G. Lagrange,} \\ {Politecnico di
Torino.}}
\title{Exchangeable Bernoulli distributions:\\
 high dimensional simulation, estimate and testing}
\maketitle

\begin{abstract}

We explore the class of  exchangeable Bernoulli distributions building on their geometrical structure. 
Exchangeable Bernoulli probability mass functions are points in a convex polytope and we have found analytical expressions for their  extremal generators.  The geometrical structure  turns out to be crucial to simulate high dimensional and negatively correlated binary data. Furthermore, for a wide class of  statistical indices and measures of  a probability mass function we are able to find not only their sharp bounds  in the class, but  also   their distribution across the class.
 Estimate and testing are also addressed.

\end{abstract}

%% Here are the title, author names and addresses

\section{Introduction}

De Finetti's representation theorem  asserts that if we have an infinite sequence of exchangeable Bernoulli variables, they can be seen as a mixture of independent and identically distributed Bernoulli variables. Using this representation it is possible to easily  estimate, test and simulate exchangeable Bernoulli variables also in high dimension, one  reason why they are widely used in applications. The only drawback of De Finetti's representation is that it requires an infinite sequence. A finite form of De Finetti's theorem has been given in \cite{diaconis1977finite}, based on the geometrical structure of the class $\EEE$ of $d$-dimensional exchangeable Bernoulli variables. In fact, $\EEE$ is proven to be a $d$-dimensional simplex and therefore each probability mass function in $\EEE$ has a unique representation as a mixture of the $d+1$ extremal points.
This result has been extended in \cite{fontana2018representation} where it is proved that the class of multivariate Bernoulli probability mass function with some  given moments is a convex polytope, i.e. a convex hull of extremal points. Furthermore, \cite{fontana2020model}  provides an analytical expression of the extremal probability mass functions  under exchangeability.
The representation of exchangeable Bernoulli  probability mass functions  as points in a convex polytope and the ability of explicitly finding  the extremal points are key steps in simulating also in high dimension, estimating and testing.

In this paper we study the statistical properties of  $\EEE$  and  $\EEE(p)$, i.e. the class of $d$-dimensional exchangeable Bernoulli probability mass functions with given mean $p$, building on their  geometrical structure. We show that  we can  find the values of a wide class of  measures as convex combinations of their values on the extremal points. As a consequence we are not only able to find  their extremal values on the class, but we can also numerically find their distribution across the class.

Two important issues in the statistical literature are the simulation of high dimensional  binary data with given correlation  and the simulation of negative correlated binary data. The geometrical structure of $\EEE(p)$ allows us to easily construct parametrical families of probability mass functions able to cover the whole correlation range.  We can therefore  select a multivariate Bernoulli probability mass function  with any given correlation in the whole range of admissible correlations. This overcome the limit of the models used in the literature to simulate exchangeable binary data, that only cover positive correlations. This is not a big issue in high dimension, since at the limit negative correlation in not possible, but it comes out to be a limitation for lower dimensions: as an example the three dimensional exchangeable Bernoulli random variables with $p=\frac{1}{3}$ admit negative correlations up to $\rho=-\frac{1}{2}$.
Furthermore, in lower dimensions the geometrical structure also turns out to be crucial  to perform  uniform sampling from convex polytope, that is limited only by the amount of computational
effort required.
Uniform sampling can be used to  find the distribution across the class of   general  statistical indices, that cannot be expresses as combinations of their extremal values.

The explicit form of the extremal points is important in estimate and testing, which are also addressed.
We find   the maximum likelihood estimator for a probability mass function in the classes  $\EEE$  and  $\EEE(p)$  and provide a generalized likelihood test for the null hypothesis of exchangeability or exchangeability with a given mean.

The results presented can be extended to the more general framework of partially exchangeable multivariate Bernoulli distributions. To give an overall idea, we show that partially exchangeable Bernoulli distributions can also be seen as points in a convex polytope, but we leave their investigation to future research.

The paper is organized as follows. Section \ref{not} introduces the polytope of exchangeable Bernoulli distributions and studies the distribution of statistical indices and measures across the class. Section \ref{simulation} addresses high dimensional simulations and discusses some applications. Section \ref{E-T} finds the maximum likelihood  estimator of exchangeable distributions using the representation of a probability mass function as linear combinations of the extremal points and provides a generalized likelihood ratio test for exchangeability. Section \ref{P-exch} opens the way to the generalization of this work to partially exchangeable Bernoulli distributions.

\section{\protect\bigskip Exchangeable Bernoulli distributions \label{not}}

Let $\mathcal{B}_d$ and  $\mathcal{E}_d\subset \mathcal{B}_d$ be the classes of $d$%
-dimensional  Bernoulli distributions and of $d$%
-dimensional exchangeable Bernoulli distributions, respectively. Let $\mathcal{E}_d(p)\subset \EEE$ be
the class of exchangeable Bernoulli distributions with the same Bernoulli
marginal distributions $B(p)$. If $\boldsymbol{X}=(X_1, \dots, X_d)$ is a
random vector with joint distribution in $\mathcal{E}_d$, we denote

\begin{itemize}

\item its cumulative distribution function (cdf) by $F$ and its probability
mass function (pmf) by $f$;

\item the column vector which contains the values of $f$ over $\mathcal{X%
}_d:=\{0, 1\}^d$, by $\ff^{\chi}=(f^{\chi}(\boldsymbol{x}):\boldsymbol{x}\in\mathcal{X}_d)$; we make the non-restrictive hypothesis that the set $\mathcal{X%
}_d$ of $2^d$ binary vectors is ordered according to the
reverse-lexicographical criterion. For example $\mathcal{X}_2=\{00, 10, 01,
11\}$ and $\mathcal{X}_3=\{000, 100, 010, 110, 001, 101, 011, 111\}$. We assume that vectors are column vectors;
%\item the marginal cumulative distribution function and the marginal mass function of $X_i$ by $F_{p,i}$ and $f_{p,i}$  respectively,  $i=1,\ldots,d$;
%\item the values $f_{p,i}(0) \equiv F_{p,i}(0)$ and $f_{p,i}(1)$ by $q$ and $p$ respectively.

\item the expected value of $X_{i}$ as $p$, $\expval[X_i]=p,\,\,\,
i=1,\ldots, d$  and we denote $q=1-p$; %

\item by $\mathcal{P}_d$ the set of permutations on $\{1,\ldots,
d\}$.

%\item since $f_{p}(\boldsymbol{x})=f_{p}(\sigma (\boldsymbol{x}))$ for any $%
%\sigma \in \mathcal{P}_d$ we write $f_{i}:=f_{p}(\boldsymbol{x})$ if $%
%\boldsymbol{x}=(x_{1},\ldots ,x_{d})\in \mathcal{D}_{d}$ and $%
%\#\{x_{i}:x_{i}=1\}=i$. We also denote with $\mathcal{D}_{d}^{j}=\{%
%\boldsymbol{x}\in \mathcal{D}_{d}:\#\{x_{i}:x_{i}=1\}=j\}$

%\item we denote by $\mathcal{C}_d^j$ the combinations of $j$ elements from $%
%\design$, $\#\mathcal{C}^j_d=\binom{d}{j}$.
\end{itemize}

%Given two matrices $A\in \mathcal{M}(n\times m)$ and $B\in \mathcal{M}%
%(d\times l)$ the matrix $A\otimes B=((a_{ij}B)_{1\le i\le n, 1\le j\le
%m})\in \mathcal{M}(nd\times ml)$ indicates their Kronecker product and $%
%A^{\otimes n}$ is $\underbrace{A\otimes \ldots \otimes A}_{n \text{ times}}$.

%We denote the transpose of vector $%\XX$ by $\XX^{T}$.

Let us consider a pmf $f\in \mathcal{E}_d$. Since, by exchangeability,  $f(\boldsymbol{x}%
)=f(\sigma (\boldsymbol{x}))$ for any $\sigma \in \mathcal{P}_d$, any
mass function $f$ in $\mathcal{E}_d$ defines $%
f_{i}:=f(\boldsymbol{x})$ if $\boldsymbol{x}=(x_{1},\ldots ,x_{d})\in
\mathcal{X}_d$ and $\#\{x_{j}:x_{j}=1\}=i$, $i=0,1,\ldots,d$. Therefore we identify a mass
function $f$ in $\mathcal{E}_d$ with the corresponding vector $%
\boldsymbol{f}:=(f_0, \ldots, f_d)$.

Let $\mathcal{S}_d(p)$ be the class of distributions $p_Y$ of   $Y=\sum_{i=1}^dX_i$ with $\boldsymbol{X}\in \mathcal{E}%
_d(p) $. The pmf $p_Y$ is a discrete distribution on  $\{0,\ldots,
d\}$. Let $P(Y=j)=p_Y(j)=p_j$ and $\boldsymbol{p}_Y=(p_0,\ldots, p_d)$.

The map:
\begin{equation}  \label{map0}
\begin{split}
H: \mathcal{E}_d(p)&\rightarrow \mathcal{S}_d(p) \\
f_{j} &\rightarrow p_j={\binom{d}{j}}f_j.
\end{split}%
\end{equation}
is a one-to-one correspondence between $\mathcal{E}_d(p)$ and $\mathcal{S}%
_d(p)$. %Its inverse is
%\begin{equation}\label{map}
%\begin{split}
%E^{-1}: &\mathcal{S}_d(p)\rightarrow \EEE(p)\\
%&p_{j} \rightarrow f_j=\frac{p_{j}}{\binom{d}{j}}.
%\end{split}
%\end{equation}
%
Therefore we have
\begin{equation}  \label{map}
\begin{split}
\mathcal{E}_d(p)&\leftrightarrow \mathcal{S}_d(p)
\end{split}%
\end{equation}
The paper \cite{fontana2020model} also proves that the class of distributions $\mathcal{S}_d(p)$ coincides
with the entire class of discrete distributions with mean $dp$, say  $\SSS(p)\equiv \DDD(dp)$.

Therefore the three classes $\mathcal{E}_d(p)$, $\mathcal{S}_d(p)$ and $%
\mathcal{D}_d(dp)$ are essentially the same class, i.e.
\begin{equation}  \label{map1}
\begin{split}
\mathcal{E}_d(p)&\leftrightarrow \mathcal{S}_d(p)\equiv\mathcal{D}_d(dp).
\end{split}%
\end{equation}
Thanks to the above result
we can look for the generators of $\mathcal{D}_d(dp)$ to find the generators of $\mathcal{S}_d(p)$ which  are in one-to-one relationship with the
generators of $\mathcal{E}_d(p)$. This approach simplifies the
search.

The relationship between exchangeable and discrete distributions is more general, since in the same way it can be proved that
\begin{equation}  \label{map2}
\begin{split}
\mathcal{E}_d&\leftrightarrow \mathcal{S}_d\equiv\mathcal{D}_d.
\end{split}%
\end{equation}

  Furthermore, by
exchangeability the moments depend only on their order, we therefore use $%
\mu_{{\alpha}}$ to denote a moment of order $\alpha=\text{ord}(\boldsymbol{%
\alpha})=\sum_{i=1}^d\alpha_i$, where $\boldsymbol{\alpha}\in \mathcal{X}_d$%
. For example we have $\mu_1=p$. %
We also observe that the Pearson's correlation $\rho$ between two Bernoulli variables $%
X_i \sim B(p)$ and $X_j \sim B(p)$ is related to the second-order moment $%
\mu_{2}= \expval[X_i
X_j]$ as follows
\begin{equation}  \label{eq:rho_e12}
\mu_2=\rho pq+p^2.
\end{equation}
\begin{comment}
From $-1\leq \rho \leq 1$ we get

\begin{equation}\label{eqrho}
p^2 - pq\leq \expval[X_iX_j] \leq p^2+pq.
\end{equation}
\end{comment}

For the sake of simplicity we write $\XX\in \EEE$ or  $\XX\in \EEE(p)$ meaning that its distribution belongs to $\EEE$ or $\EEE(p)$, respectively. Analogously for $Y\in \mathcal{S}_d$ or $Y\in \mathcal{S}_d(p)$.

\subsection{Polytope of Exchangeable Bernoulli distributions}
We recall that a polytope (or more specifically a $d$-polytope)
is the convex hull of a finite set of points in $\RR^d$ called the extremal points of the polytope. We say that a set of $k$ points is affinely independent
if no one point can be expressed as a linear convex combination of the others. For example, three points are affinely independent if they are not on the same line, four
points are affinely independent if they are not on the same plane, and so on. The convex hull of $k+1$  affinely
independent points is called a simplex or $k$-simplex. For example, the line segment joining two points is a
1-simplex, the triangle defined by three points is a 2-simplex, and the tetrahedron defined by four points is a
3-simplex.

The class of discrete distributions $\pp=(p_0,\ldots, p_d)$ on $\{0,\ldots, d\}$ is the  $d$-simplex   $\Delta_d=\{\pp: p_i\geq 0, \sum_{i=0}^dp_i=1\}$, with extremal points $\gg_j=(0,\ldots,1,\ldots, 0)$, $j=0,\ldots, d$.
By means of  the equivalence  $\mathcal{S}_d\equiv \mathcal{D}_d$ and the  map $H$ we have  that the class  $\EEE$ is a $d$-simplex. In 2006, \cite{diaconis1977finite} proved that $\EEE$ has $d+1$ extremal points $\gg'_0, \ldots, \gg'_{d}$, where $\gg'_j=(g'_j(\xx); \xx\in \chi_d)$  is the measure
\begin{equation*}
g'_{j}(\xx)=\left\{
\begin{array}{cc}
\frac{1}{\binom d j} &  \text{if}\,\,\,  \#\{x_h: x_h=1\}=j \\
0 & \text{otherwise}%
\end{array}
\right..
\end{equation*}
The same can be obtained by inverting the map $H$.

In \cite{fontana2020model}, the authors proved that the class of discrete distributions on $\{0,\ldots, d\}$ with mean $p$,
 $\mathcal{S}_{d}(p )$, is a $d$-polytope, i.e. for any $Y\in\mathcal{S}_{d}(p )$  there
exist $\lambda_1, \ldots, \lambda_{n_p}\geq 0$ summing up to 1 and $\boldsymbol{r}_{j}\in\mathcal{S}_{d}(p )$  such that
\begin{equation}\label{cone}
\boldsymbol{p}_{Y}=\sum_{j=1}^{n_{p}}\lambda _{j}\boldsymbol{r}_j.
\end{equation}
We call  $\boldsymbol{r}_{j}=(r_{j}(0), \ldots, r_j(d))$, $j=1,\ldots, n_p$  the extremal points or the extremal densities.
The $d$-polytope is the set of solutions of a linear system. As a consequence we can find the support of the extremal rays and also
their analytical expression. The following two propositions are proved in \cite{fontana2020model}.

\begin{proposition}
\label{multinulli} Let us consider the linear system
\begin{equation}
A\boldsymbol{z}=0,\,\,\, \boldsymbol{z}\in \mathbb{R}_+^{d+1}  \label{system}
\end{equation}%
where $A$ is a $m\times (d+1)$ matrix, $m\leq d$ and $\rank(A)=m$. The
extremal rays of the system \eqref{system} have at most $m+1$ non-zero
components.
\end{proposition}

\begin{proposition}
\label{binu} %Case one: $pd$ is not integer.
The extremal rays $\boldsymbol{r}_{j}$  in \eqref{cone} have support on two points $(j_1, j_2)$  with $j_1=0,1,\ldots, j_1^{M}$, $j_2=j_2^m, j_2^m+1, \ldots, d$, $j_1^M$ is
the largest integer less than $pd$ and $j_2^m$ is the smallest integer
greater than pd.  They are
\begin{equation}  \label{binul}
r_{j}(y)=\left\{
\begin{array}{cc}
\frac{j_2-pd}{j_2-j_1} & y=j_1 \\
\frac{pd-j_1}{j_2-j_1} & y=j_2 \\
0 & \text{otherwise}%
\end{array}
\right..
\end{equation}

If $pd$ is integer the extremal rays contain also
\begin{equation}  \label{onenul}
r_{pd}(y)=\left\{
\begin{array}{cc}
1 & y=pd \\
0 & \text{otherwise}%
\end{array}
\right..
\end{equation}
 If $pd$ is not integer there are $n_p=(j_1^M+1)(d-j_1^M)$ ray
densities.
If $pd$ is integer there are $n_p=d^2p(1-p)+1$ ray densities.

\end{proposition}

Using the equivalence $\mathcal{S}_d(p)\equiv \mathcal{D}_d(pd)$  a pmf in $\mathcal{S}_d(p)$ is a pmf on $%
\{0,\ldots,d\}$ with mean $pd$. Thanks to the map $H$ in Eq. \eqref{map0}
this is also equivalent to state that
$ \mathcal{E}_{d}(p )$    is a $d$-polytope, i.e. for any $\XX\in\mathcal{E}_{d}(p )$     there
exist $\lambda_1, \ldots, \lambda_{n_p}\geq 0$ summing up to 1 and  $\boldsymbol{e}_{i}\in\mathcal{E}_{d}(p )$ such that
\begin{equation}\label{cxray}
\boldsymbol{f}=\sum_{j=1}^{n_{p}}\lambda _{j}\boldsymbol{e}_{j}.
\end{equation}
We  call $\boldsymbol{e}_{j}$  the extremal points of $\mathcal{E}_{d}(p )$. The map $H$ allows us to explicitly find the extremal points $\ee_j$. They are:

\begin{equation}  \label{onenul}
e_{j}(\xx)=\left\{
\begin{array}{cc}
\frac{r_j(k)}{\binom d k} &  \text{if}\,\,\,  \#\{x_h: x_h=1\}=k \\
0 & \text{otherwise}%
\end{array}
\right..
\end{equation}

We denote by $R_{j}$ and $R_{pd}$ the random variables whose pmfs are
${r}_{j}$ and ${r}_{pd}$ respectively and  by $\EE_{j}$ and $\EE_{pd}$ the random variables whose pmfs are
${e}_{j}$ and ${e}_{pd}$ respectively. We will
refer to ${r}_{j}$,  ${r}_{pd}$,  ${e}_{j}$ and ${e}_{pd}$ as %
extremal ray densities. If it clear from the context, we will omit
extremal for the sake of simplicity.

The following proposition is a consequence of the geometrical structure of the class of exchangeable distributions and their sums. It allows us to have an analytical expression for a wide class of statistical  indices defined as functionals on $\EEE$ and $\SSS$, such as all the moments of the Bernoulli exchangeable distributions.
\begin{proposition}\label{hull-phi}
\begin{enumerate}
\item \label{pa}
Let $\XX\in \EEE$ [$\XX\in \EEE(p)$] and let be $\phi_d:\RR^d\rightarrow \RR$ a measurable function. Then
\begin{equation}
E[\phi_d(\XX)]=\sum_{i=1}^m\lambda_iE[\phi_d(\EE_j)],
\end{equation}
where $\EE_1, \ldots, \EE_m$ are the extremal rays of $ \EEE$ [$ \EEE(p)$].
\item \label{pb}
Let $Y\in \SSS$ [$Y\in \SSS(p)$] and let be $\phi:\RR\rightarrow \RR$ a measurable function. Then
\begin{equation}
E[\phi(Y)]=\sum_{i=1}^m\lambda_iE[\phi(R_j)],
\end{equation}
where $R_1, \ldots, R_m$ are the extremal rays of $ \SSS$ [$ \SSS(p)$].
\end{enumerate}
\end{proposition}
\begin{proof}
We only proof part \ref{pa}, because part \ref{pb} is analogous.

It holds

\begin{equation}
E[\phi_d(\XX)]=\sum_{\xx\in \chi_d}\phi_d(\xx)f(\xx)=\sum_{\xx\in \chi_d}\phi_d(\xx)\sum_{i=1}^m\lambda_jr_j(\xx)=\sum_{i=1}^m\lambda_iE[\phi_d(\EE_j)].
\end{equation}
\end{proof}
 Functionals  defined on a class $\mathcal{F}_0$ of $d$-dimensional distributions,  $\Phi: \mathcal{F}_0\rightarrow \RR$, $\Phi(f)=E[\phi_d(\XX_f)]$, where $\XX_f\in \mathcal{F}_0$ has pmf $f$, are commonly used in applications to define measures of risk \cite{wang2020risk}.
We call expectation measures the  measures defined by expectations and by their one-to-one transformations.
Proposition \ref{hull-phi} states that  measures defined by expectation of the mass functions in a given class are themselves a convex polytope  whose extremal points are the measures evaluated on  the extremal rays of the class.
Therefore, they have bounds  on the  extremal points.
 Examples of such functionals in our framework are:
\begin{enumerate}
\item Moments and cross moments of distributions in $\EEE$ or $\EEE(p)$ and moments of discrete distributions in  $\mathcal{D}_d$ or $\mathcal{D}(dp)$.

\item The entropic risk measure on $\SSS$ or  $\SSS(p)$ for $\gamma\in (0,\infty)$:
\begin{equation}
\Gamma(f)=\frac{1}{\gamma}\log(E[e^{-\gamma Y}]),
\end{equation}
where $Y$ has pmf $f$.
\item Excess loss function on $\SSS$ or  $\SSS(p)$, defined by
\begin{equation}
\Phi(Y)=E[(Y-k)^+], \,\,\, k\in \RR,
\end{equation}
where $(x-k)^+=\max\{x-k, 0\}$.
\item Von Neumann-Morgestern expected utilities on $\SSS$ or  $\SSS(p)$. These are expectation measures, where the function  $\phi_d$  is some increasing utility function.
\end{enumerate}

Notice that also the entropic risk measure has its bounds on the extremal points, since the logarithm is a monotone transformation.

We focus on the cross $\alpha$-moments  $\mu_{\alpha}=E[X_1\cdots X_{\alpha}]$. We have
\begin{equation*}
\begin{split}
\mu_{\alpha}=\sum_{k=\alpha}^d\frac{\binom{d-\alpha}{k-\alpha}}{\binom{d}{k}}p_k,
\end{split}
\end{equation*}
where $p_k=P(\sum_{i=1}^dX_i=k)$.
We recall that  the second order moment of the class $\SSS(p)$
are given by
\begin{equation}  \label{smu2}
E[Y^2]=E[(X_1+\ldots+X_d)^2]=pd+d(d-1)\mu_2.
\end{equation}
In the next section, using Proposition \ref{hull-phi} and convexity,  we numerically  find the distribution of the above measures across the class of pmf where they are defined.

\subsection{Expectation measure distributions}  \label{sec:exp}

We want to determine the distribution $E[\phi_d(\XX)]$ where $\XX$ is a random variable which has been chosen uniformly at random from $\EEE(p)$ or $ \EEE$. Since classes of  multivariate Bernoulli distributions  with pre-specified moments, as e.g. $\EEE(p)$, are polytopes,  the representation of each mass function as  a convex linear combination of the extremal points is not unique.
Therefore  we have to perform a triangularization of $\EEE(p)$. To reduce dimensionality we aim to work on  $\mathcal{D}_d(dp)$. 
This can be done using the map $H$ in \eqref{map0} under the exchangeability condition  $f(\xx)=f(\sigma(\xx))$ for $f\in \EEE(p)$. As a consequence  the distribution of   $E[\phi_d(\XX)]$ can be studied in  $\mathcal{D}_d(dp)$ if $\phi_d(\xx)=\phi_d(\sigma(\xx))$;  we make this assumption in this section. We  perform a triangularization of $\mathcal{D}_d(dp)$ that  is equivalent  to perform a triangularization of the polytope $\mathcal{C}=\{\pp: p_i\geq 0, \sum_{i=0}^d p_i=1, \sum_{i=0}^d i p_i= dp\}$.
 The dimension of $\mathcal{C}$ is $d-1$ because $\mathcal{C}$ is defined by two constraints.
We can partition $\mathcal{C}$ into simplices  $\mathcal{T}_i, \, i\in \mathcal{I}$ (e.g. using a Delaunay triangulation)
\begin{equation}\label{Triang}
\mathcal{C}=\bigcup_{i \in \mathcal{I}} \mathcal{T}_i,
\end{equation}
where $\mathcal{I}$ is a proper set of indices, and $\mathcal{T}_i \cap \mathcal{T}_j = \emptyset$ for $i \neq j$. We observe that from a geometric point of view the intersection between two simplices $\mathcal{T}_i $ and $\mathcal{T}_j$ is not empty, being the common part of their borders. But this common part has zero probability of being selected and so we can neglect it assuming that each $\mathcal{T}_i,\,\,  i\in \mathcal{I}$ coincides with its interior part.

Let $\phi(y)=\phi_d(\xx)$, where $y=\sum_{i=1}^d x_i$, $y=\{0,\ldots, d\}$. Let us denote by $F_\phi$ the distribution of $E[\phi(Y)]$, where $Y$ is a random variable with pmf $p_Y$  from $\SSS(p)$ or $ \SSS$. We get
\begin{equation}\label{prob_tot}
F_\phi(t)=P(E[\phi(Y_p)] \leq t)=\sum_{i \in \mathcal{I}}P( \mathcal{T}_i) P(E[\phi(Y_p)] \leq t| \mathcal{T}_i).
\end{equation}

If we assign a uniform measure on the space $\EEE(p)$ the probability $P( \mathcal{T}_i) $ of sampling a probability mass function in the simplex $\mathcal{T}_i$ is simply the ratio between the volume of  $\mathcal{T}_i$ and the total volume of $\mathcal{C}$, i.e.
\begin{equation} \label{volsim}
P(\mathcal{T}_i)=\frac{\text{vol}(\mathcal{T}_i)}{ \text{vol}(\mathcal{C})}.
\end{equation}
The volume of $\mathcal{C}$ can be easily computed because the volume  of an $n$-simplex in $n$-dimensional space with vertices $(v_0, \ldots, v_n)$ is
\[
{\displaystyle \left|{1 \over n!}\det {\begin{pmatrix}v_{1}-v_{0},&v_{2}-v_{0},&\dots ,&v_{n}-v_{0}\end{pmatrix}}\right|}
\]
where each column of the $n \times n$ determinant is the difference between the vectors representing two vertices \cite{stein1966note}.

The probability $P(E[\phi(Y)] \leq t| \mathcal{T}_i)$ is the ratio between the volume of the region  $\mathcal{R}_{i,t}=\{p_Y \in \mathcal{T}_i: E[\phi(Y)] \leq t\}$ and the volume of $\mathcal{T}_i$, i.e.
\begin{equation} \label{volfrustum}
P(E[\phi(Y)] \leq t| \mathcal{T}_i)=\frac{\text{vol}(\mathcal{R}_{i,t})}{\text{vol}(\mathcal{T}_i)}.
\end{equation}
The computation of the volume of $\mathcal{R}_{i,t}$ will depend on the definition of $\phi$ in the expectation measure $E[\phi(Y)]$.

We now consider the $k$-order moments $\mu_k^{(Y)}$ of the random variable $Y=X_1+\ldots+X_d$ whose pmf is denoted by $p_Y$
\[
\mu_k^{(Y)}=E[Y^k]=\sum_{i=0}^d i^k p_Y(i)=\sum_{i=0}^d i^k p_i
\]
From Eq. \eqref{map0} we have $p_Y(i)={\binom{d}{i}}f_i$ and  $f_{i}:=f_{p}(\boldsymbol{x})$
for $\boldsymbol{x}=(x_{1},\ldots ,x_{d})\in \mathcal{X}_d$ and $\#\{x_{k}:x_{k}=1\}=i$.

The pmf $p_Y$ will lie in exactly one of the simplices $\mathcal{T}_i, i \in \mathcal{I}$. Let's denote this simplex by $\mathcal{T}_{i^\star}$. We can write $p_Y=\sum_{j \in \mathcal{J^*}} \lambda_j r_j$, where $\mathcal{J^*}$ is the set of indexes that defines the subset of rays which generate the simplex $\mathcal{T}_{i^\star}$, i.e. $\mathcal{T}_{i^\star}=\text{simplex}(\{r_j: j \in \mathcal{J^*}\})$. As a corollary of Proposition \ref{hull-phi} we can write
\[
\mu_k^{(Y)}=\sum_{j \in \mathcal{J^*}} \lambda_j \mu_k^{(j)}
\]
where $\mu_k^{(j)}$ are the $k$-moments of the ray random variables $R_j$,  $ \mu_k^{(j)}=E[R_j^k]$.
For $k$-order moments the region $\mathcal{R}_{i,t}=\{p_Y \in \mathcal{T}_i: \mu_k(Y) \leq t\}$  is  the subset of the standard simplex defined as $\{ (\lambda_j; j\in \mathcal{J^*}) : \lambda_j \geq0, \sum_{j\in \mathcal{J^*}} \lambda_j=1, \sum_{j \in \mathcal{J^*}} \lambda_j \mu_k^{(j)}\leq t\}$. For $k$-order moments the ratio of the volumes in Eq. \eqref{volfrustum} can be computed using an exact and iterative formula, see \cite{varsi1973multidimensional} and \cite{cales2018practical}.

The same methodology can be applied also to other measures defined as function of expected values like the entropic risk measures.

\section{High dimensional simulation}\label{simulation}
The representation of Bernoulli pmfs as convex combinations of ray densities allows us to sample from high dimensional  exchangeable Bernoulli distributions in $\EEE(p)$. The simplest way to do that is to generate directly from a ray density. Ray densities are extremal and have support on two points,  we can also use combinations of rays  to choose a distribution in the interior of the polytope. As an example we could choose $\lambda_i=\frac{1}{n_p}$  in \eqref{cxray}. Such a choice identifies  a pmf inside the polytope.  Nevertheless, the ray densities allow us to simulate from a family of distributions that cover the whole range spanned by a measure of dependence.
Let $M$ and $m$ the maximum and minimum values of an expectation measure. If $r_M$ and $r_m$ are two corresponding  ray densities, the parametrical family $f=\lambda r_m+(1-\lambda)r_M$ span the whole range $[m, M]$. We can therefore use this family to simulate a sample of binary data from a pmf with a given value of the measure or we can simulate binary variables with different values of the measure, by moving $\lambda$. An important example is correlation.
The problem to simulate from multivariate Bernoulli distributions with given correlations and in particular  negative correlations  is of interest in many applications and is addressed in the statistical literature, see \cite{oman2009easily}.
The geometrical structure of $\EEE(p)$ allow us to solve this problem for exchangeable random vectors. In fact, we can choose a pmf with the required correlation simply by finding two ray densities in $\EEE(p)$ with the  minimum and maximum of allowed correlations, say  $r_{\rho_m}$  and $r_{\rho_M}$ respectively. The pmfs in the family $r_{\lambda}=\lambda r_{\rho_m}+(1-\lambda)r_{\rho_M}$ span then whole correlation range.
Simulation of a pmf in this family is easy and can be performed in high dimension.
We provide  general  algorithm to simulate  high dimensional exchangeable  Bernoulli   distribution from $\EEE(p)$, given the vector $\boldsymbol{\lambda}=(\lambda_1, \ldots, \lambda_{n_p})$ in \eqref{cone} or equivalently in \eqref{cxray}.  The vector $\boldsymbol{\lambda}$ can be chosen as we discussed above or it can be randomly chosen
  by giving a distribution $P_{\Lambda}$ on the simplex $\Lambda=\{ \boldsymbol{\lambda}\in \RR^{n_p}: \sum_{i=1}^{n_p}\lambda_i=1, \lambda_i>0, i=1,\ldots, n_p\}$ in \eqref{cone}.
Once $\boldsymbol{\lambda}$ is selected, it represents a probability distribution on the set of ray densities, i.e. $\lambda_j$ is the probability to extract the ray density $r_j$.
Given the pmf $\boldsymbol{\lambda}$ on the set of extremal rays, the following algorithm allows us to simulate in high dimension:
\begin{table}[h!]\label{Algo}
%\begin{center}
\begin{tabular}{l}
\hline
{\bf Algorithm}  \\
\hline
{\bf Input}: the expected value $p$, the dimension $d$, the vector $\lambda\in \Lambda$. \\
\hline
1) Select a   ray density $r_j$ with probability  $\lambda_i$; $r_j$ has support on $j_1, j_2$ \eqref{binul}.\\
2) Select $j_*\in \{j_1, j_2\}$  with probability $r_j(j_1)$ and $r_j(j_2)$ in \eqref{binul} respectively.\\
3) Select a binary vector with $j_*$ ones among the combinations $\binom {d}{ j_*}$.\\
\hline
{\bf Output}: One realization of a $d$ dimensional binary variable with pmf in \eqref{cxray}. \\
\hline
\end{tabular}
%\end{center}
\end{table}

We observe that the Algorithm does not require to store any big structure and then it can be easily used for large $d$, e.g. $d=10^5$.

This case is interesting because the families of multivariate Bernoulli variables commonly used for simulation of exchangeable binary variables incorporate only positive correlation. We consider here  two families of exchangeble Bernoulli models used in the literature  to simulate high dimensional binary data. The first family  is proposed in \cite{jiang2020set} and we term it family of  one-factor models, taking the name from the  one-factor models used in credit risk, that have a similar  dependence structure.   The second family  is the mixture model based on De Finetti's representation theorem.

The construction proposed in \cite{jiang2020set} provides an algorithm to generate binary data with given marginal Bernoulli distributions with means $(p_1,\ldots, p_d)$ and exchangeable dependence structure, meaning that they are equicorrelated. However, by assuming that the marginal parameters are equal to a common parameter $p$, their construction gives a vector $\XX\in \EEE(p)$. This is the case considered here. We therefore define the multivariate Bernoulli variable in this framework.
Let
\begin{equation}
X_i=(1-U_i)Y_i+U_iZ, i=1,\ldots, d,
\end{equation}
where $U_i\sim B(\sqrt{\rho})$, $Y_i\sim B(p)$, $i=1,\ldots, d$ and $Z\sim B(p)$ and they are independent. We say that $\XX=(X_1, \ldots, X_n)$ and its pmf $f\in \EEE(p)$ have a one-factor structure.
Clearly, $\XX$ is exchangeable, have distribution $B(p)$ and correlation $\rho$. By construction we have $\rho\geq 0$ and the case $\rho=0$ implies that $U_i$, $i=1,\ldots, n$ put all the mass on $0$.

According to De Finetti's Theorem if $f\in \EEE(p)$ is the pmf of a random vector $(X_1, \ldots, X_d)$ extracted from an exchangeable sequence then $f$ has the representation

\begin{equation*}
f(j)=\binom{d}{j}\int_{0}^{1}p^{k}(1-p)^{d-k}d\Psi (p),
\end{equation*}%
where $\Psi (p)$ is a pdf on $[0,1]$. Clearly, these vectors can only have positive correlations.
One of the most used mixed Bernoulli model is the $\beta $-mixing models, where:
 $\Psi \sim \beta (a,b)$ is the mixing variable. In this case we have
\begin{equation*}
\begin{split}
& p=E[\Psi ] \\
& \mu _{2}=E[\Psi ^{2}].
\end{split}%
\end{equation*}%
Therefore we estimate the $\beta $ parameters $a$ and $b$ by solving the
equations
\begin{equation*}
\begin{split}
& p=\frac{a}{a+b} \\
& \mu _{2}=\frac{a(a+1)}{(a+b)(a+b+1)}.
\end{split}%
\end{equation*}%
For this model $\rho =0$ is not admissible, therefore the model cannot include independence.
Notice that the
$\beta $ model has two parameters and, chosen a $\beta$-mixing model in $\EEE(p)$ it can be parametrized by  $\rho $.

In \cite{fontana2020model} the authors analytically found the correlation bounds for exchangeable pmf for each dimension $d$ and the minimum attainable correlation.
The bounds for correlations are:
\begin{itemize}
\item if $pd$ is not integer
\begin{equation}\label{corint}
\frac{\frac{1}{d(d-1)}[-j_1^M(j_1^M+1)+2j_1^M pd]-p^2}{p(1-p)}\leq \rho\leq
1.
\end{equation}
\item If $pd$ is integer
\begin{equation}\label{cornint}
-\frac{1}{d-1}\leq \rho\leq 1.
\end{equation}
\end{itemize}
In both cases the minimal correlation $\rho_m$ goes to zero  if the dimension increases,  according to De Finetti's representation theorem. Therefore, the capability to generate binary data with negative correlations is more important in low dimensions, where we are also able to perform uniform sampling, as discussed in the next section.

\subsection{Uniform simulation}\label{sec:nonexp}
Let's start considering uniform sampling from $\EEE$.
From Eq.\eqref{map2} we know $\mathcal{E}_d \leftrightarrow \mathcal{S}_d\equiv\mathcal{D}_d$. It follows that sampling uniformly at random from $\mathcal{E}_d$ is equivalent to  sampling uniformly at random from $\mathcal{D}_d$ and then it  is equivalent  to  sampling uniformly at random from the $d$-simplex $\Delta_d=\{\pp: p_i\geq 0, \sum_{i=0}^d p_i=1\}$, which is a standard topic in the statistical literature.

Let's now consider uniform sampling from $\EEE(p)$.
From Eq.\eqref{map1} we know $\mathcal{E}_d(p) \leftrightarrow \mathcal{S}_d(p)\equiv\mathcal{D}_d(dp)$. It follows that sampling uniformly at random from  $\mathcal{E}_d(p)$ is equivalent to  sampling uniformly at random from $\mathcal{D}_d(dp)$ and then is equivalent  to  sampling uniformly at random from the polytope $\mathcal{C}=\{\pp: p_i\geq 0, \sum_{i=0}^d p_i=1, \sum_{i=0}^d i p_i= dp\}$. We therefore have to consider the triangularization of $\mathcal{C}$  in \eqref{Triang}. We can consider sampling from $\EEE$ as a special case of sampling from $\EEE(p)$. For this purpose when sampling from $\EEE$ we denote  the d-simplex by $\mathcal{T}_1$, $\mathcal{T}_1\equiv \Delta_d$. We have $\mathcal{I}=\{1\}$.
Uniform sampling allow us to find the empirical distribution of some statistical indices that are not expectation measures.
In fact, there are famous measures that are not expectation measures and for which  Proposition \ref{hull-phi} does not apply.
For example the $\alpha$-quantile or value at risk of a balanced portfolio of exchangeable Bernoulli variables, widely used in applications as a measure of risk.
Let $Y\in \SSS(p)$, the $\alpha$-quantile $%
{q}_{\alpha}$ at level $\alpha $ is defined by
\begin{equation*}
{q}_{\alpha}(Y)=\inf \{y\in \mathbb{R}:P(Y\leq y)\geq \alpha \}.
\end{equation*}%

The $\alpha$-quantile is not a convex measure, nevertheless \cite{fontana2020model} proved the following proposition.

\begin{proposition}
\label{VEbound}
 Let $Y\in \mathcal{S}_d(p)$. Then
\begin{equation*}
\min_{R} {q}_{\alpha}(R)\leq \text{q}_{\alpha}(Y)\leq \max_{R}
\text{q}_{\alpha}(R),
\end{equation*}
where $R$ has pmf $r$ that is a  ray densities of $\mathcal{S}_d(p) $.
%\item Let $k$ be the $\alpha $-quantile of $S_{d}$, i.e. $P(S_{d}\leq
%k)=\alpha $. Then $P(S_{d}\leq k)=\sum_{l=0}^{n_{p}}\lambda
%_{l}\sum_{h=0}^{k}r_{S_{h}}^{i}$.\textbf{pa scrivere coi var}
\end{proposition}

Another famous measure defined on classes of distribution is the entropy. For a discrete pmf $p_Y\in \SSS$ it is given by:
\begin{equation*}
E(p_Y)=\sum_{i=0}^dp_i\log p_i.
\end{equation*}

The entropy  does not satisfy Proposition \ref{hull-phi} and does not reach its bound on the ray densities.
We therefore are not able to use the geometry of a pmf in $\SSS$ to numerically find its distribution across the $\SSS$. However, we can address this goal using simulations.
For other measures for which the exact value of the ratio in \eqref{volfrustum} cannot be computed,
 an estimate of it can be simply obtained by sampling uniformly at random over  $\mathcal{T}_i$ and determining the relative frequency of the points that fall in $\mathcal{R}_{i,t}$
\[
\widehat{\left( \frac{\text{vol}(\mathcal{R}_{i,t})}{\text{vol}(\mathcal{T}_i)} \right)}=\frac{\#\{p_k \in \mathcal{R}_{i,t}, \, k=1,\ldots, N \}}{N}
\]
where $N$ is the size of the sample. In these cases an estimate $\hat{F_\phi}$ of the distribution $F_\phi$ will be obtained.

\subsection{Applications}

We present two applications where we will study different scenarios.
The first application is in dimension $d=3$, the polytope is $2$-dimensional and we can explicitly see the triangularization. The second example is performed in higher dimensions, $d=6$, where the structure of the polytope is more complex.
\subsubsection{Application 1}
We study:
\begin{enumerate}
\item the exact distribution of the second-order moment in $\mathcal{E}_3(0.4)$,  the class of exchangeable distributions of dimension $d=3$ and mean $p=0.4$;
\item  a family of pmfs for simulations that span the whole range of correlation  in $\mathcal{E}_3(0.4)$;
\item the sampling distribution of the entropy in the same class $\mathcal{E}_3(0.4)$;
\item the joint distribution of the first-order moment and  correlation in $\mathcal{E}_3$, the class  of exchangeable distributions of dimension $d=3$.
\end{enumerate}

The ray densities of $\mathcal{E}_3(0.4)$ are the columns in Table \ref{tabray}.

\begin{table}
\begin{center}\caption{Ray density of  $\mathcal{E}_3(0.4)$}
\begin{tabular}{|r |r| r| r| r| } \label{tabray}
$y$&$r_1(y)$ & 	$r_2(y)$ &	$r_3(y)$ &$r_4(y)$ \\
\hline
0&0.4 &	0.6 &	0 &	0 \\
1&0 &	0 &	0.8 &	0.9 \\
2&0.6 &	0 &	0.2 &	0 \\
3&0 &	0.4 &	0 &	0.1 \\
\end{tabular}
\end{center}
\end{table}
The ray densities are points in $\RR^4$ which lie in a subspace of dimension $d-1=3-1=2$. Using standard Principal Component
Analysis we can project these 4 points to $\RR^2$.  The points inside the polygon in Figure \ref{fig01} (left side) represent all the densities which belong to  $\mathcal{S}_3(0.4)\leftrightarrow\mathcal{E}_3(0.4)$.
%
%\begin{figure}
%%\includegraphics[width=\linewidth]{C:\\Users\\roberto\\Dropbox (Politecnico di Torino Staff)\\AnnalseRoyal\\StatCo\\tabelle\\polygone3p40}
%\includegraphics[width=\linewidth]{polygone3p40.png}
%  \caption{$\mathcal{E}_3(\frac{2}{5})$}
%  \label{fig01}
%\end{figure}

\begin{figure}[h!]
\caption{$2$-dimensional polytope $\mathcal{E}_3(0.4)$}\centering
\begin{subfigure}[b]{0.3\textwidth}
\includegraphics[width=\textwidth]{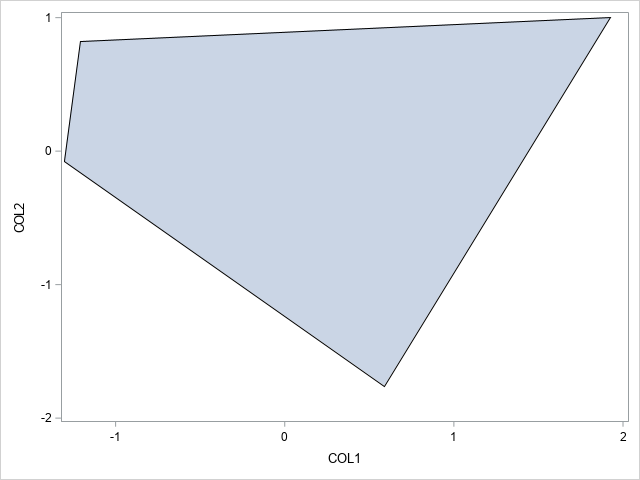}
\end{subfigure}
\begin{subfigure}[b]{0.3\textwidth}
\includegraphics[width=\textwidth]{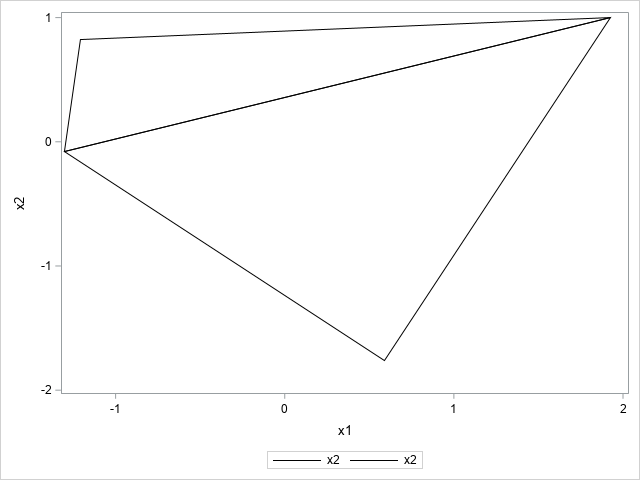}
\end{subfigure}
\label{fig01}
\end{figure}

The right side of Figure \ref{fig01}  shows the triangularization. We have two triangles: $\mathcal{T}_1$, the largest one with area $3.74$  and $\mathcal{T}_2$, the smallest one with area  $1.40$. Then, the sampling probabilities are $P(\mathcal{T}_1)=0.73$ and
$P(\mathcal{T}_2)=0.27$.

For $2$-order moments the region $\mathcal{R}_{i,t}=\{p_Y \in \mathcal{T}_i: \mu_2(Y) \leq t\}$  is  the subset of the standard simplex defined as $\{ (\lambda_j; j\in \mathcal{J^*}) : \lambda_j \geq0, \sum_{j\in \mathcal{J^*}} \lambda_j=1, \sum_{j \in \mathcal{J^*}} \lambda_j \mu_2^{(j)}\leq t\}$. We recall that in this case the ratio of the volumes in Eq. \eqref{volfrustum} can be computed using an exact and iterative formula,  see \cite{varsi1973multidimensional} and \cite{cales2018practical}.

Figure \ref{fig02} (left side) exhibits the exact numerical  cumulative distribution function (cdf) $P(\mu_2(Y) \leq t| \mathcal{T}_1)$ of $\mu_2(Y)$ across $\mathcal{T}_1$. The distribution of $\mu_2(Y)$  across $\mathcal{T}_2$ is similar. The  cdf  $F_{\mu_2}$ of $\mu_2(Y)$ across the whole polytope is obtained   by mixing the conditional cdfs as in \eqref{prob_tot}.

Figure \ref{fig02} - right side - shows the probability density function (pdf) of the mixture obtained from the  cdf of $\mu_2$ by $f_{\mu_2}(t)=\frac{F_{\mu_2}(t+\Delta)-F_{\mu_2}(t)}{\Delta}$, where  $\Delta$ has been chosen equal to  $(\max(\mu_2)-\min(\mu_2))/10000$.

\begin{figure}[h]
\caption{Distribution of the $2$-order moment across $\mathcal{E}_3(0.4)$}\centering
\begin{subfigure}[b]{0.3\textwidth}
\includegraphics[width=\textwidth]{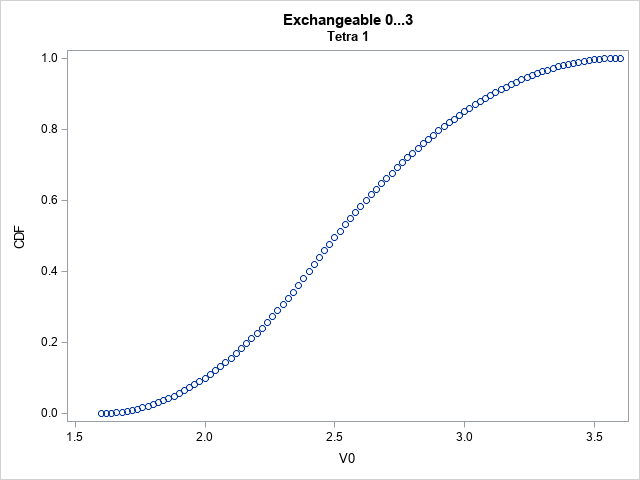}
\end{subfigure}
\begin{subfigure}[b]{0.3\textwidth}
\includegraphics[width=\textwidth]{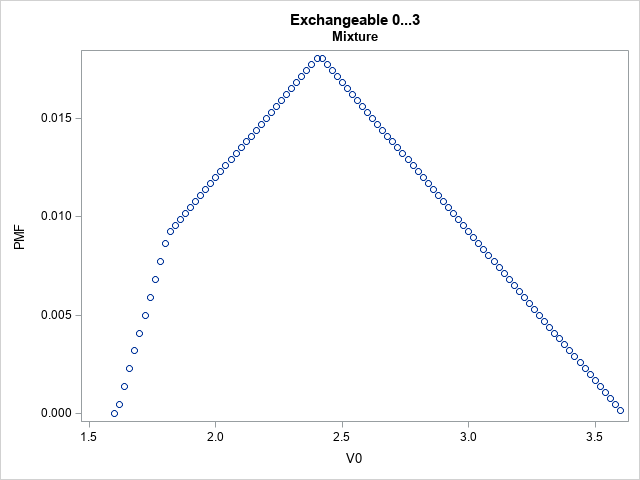}
\end{subfigure}
\label{fig02}
\end{figure}

The pmfs in the family $r_{\lambda}=\lambda r_{\rho_m}+(1-\lambda) r_{\rho_M}$ span then whole correlation range and can be simulated according to the Algorithm proposed, Figure \ref{fig03a} shows the family in the polytope $\EEE(p)$ together with the families of $\beta$-mixture models, with mean $p$ and of one-factor models, given $p$.  It is evident that our approach allows us to consider also pmfs with negative correlations (green straight line in the figure), while the other two approaches provide only positive correlations (blue and red lines in the figure). In dimension three the range of negative correlation  $[-0.39, 0]$ is wide, as evidenced in the figure.

\begin{figure}[h]
\caption{Families of distributions: $r_{\lambda}$ (straight line: yellow line are positive correlations, green line are negative correlations), $\beta$-mixing (red line) and one-factor (blue line) across $\mathcal{E}_3(0.4)$}\centering
\includegraphics[width=0.3\textwidth]{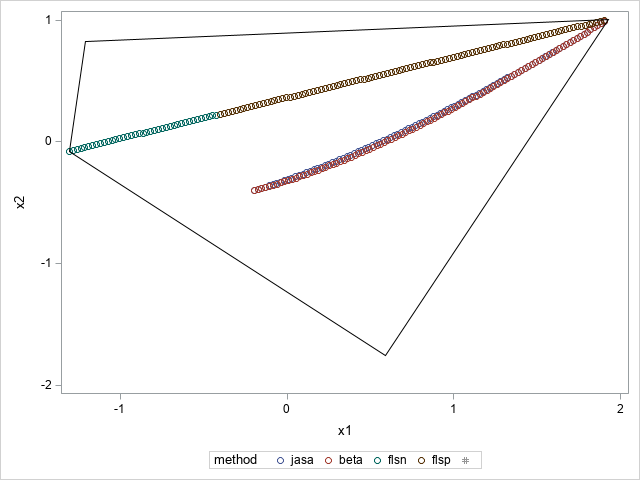}
\label{fig03a}
\end{figure}

The simulated pdf of the entropy in the  class $\mathcal{E}_3(0.4)$, where the entropy of $\XX\in \mathcal{E}_3(0.4)$  is defined to be the entropy of $Y=\sum_{i=1}^dX_i$, is found using the methodology in Section \ref{sec:nonexp} and it is shown in Figure \ref{fig03}.

\begin{figure}[h]
\caption{Empirical distribution of the entropy across $\mathcal{E}_3(0.4)$}\centering
\includegraphics[width=0.3\textwidth]{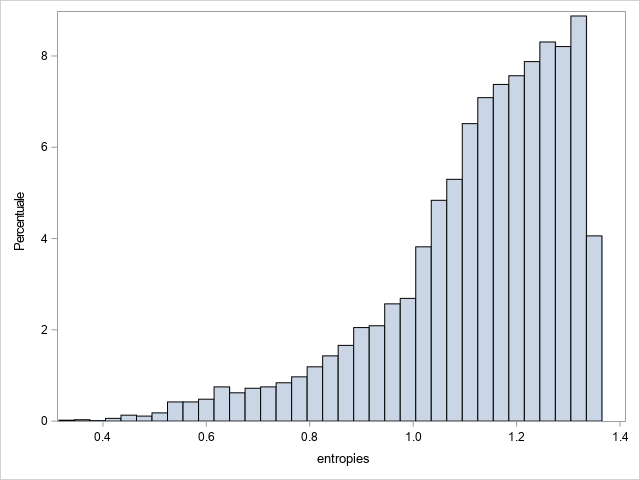}
\label{fig03}
\end{figure}

The simulated pdf can obviously also be found for the second order moment $\mu_2(Y)$. It  is shown in Figure \ref{fig04} for completeness. The simulated pdf is obviously in agrement with the exact numerical one (right side of Figure \ref{fig02}).

\begin{figure}[h]
\caption{Empirical distribution of the second order moment across $\mathcal{E}_3(0.4)$}\centering
\includegraphics[width=0.3\textwidth]{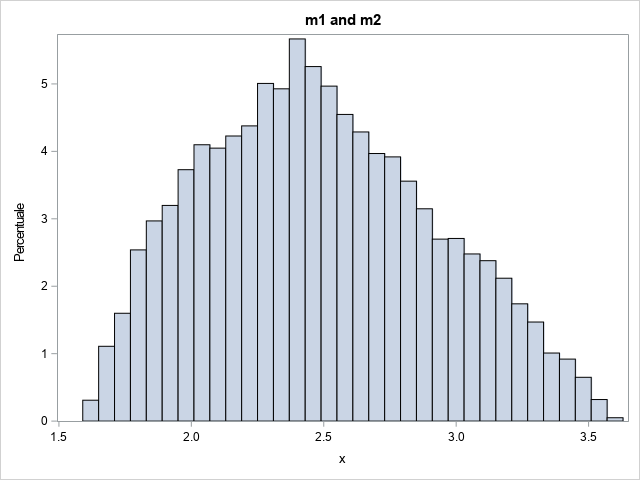}
\label{fig04}
\end{figure}

Figure \ref{fig05} shows the simulated bivariate distribution of the mean $p$ and the correlation $\rho$  across $\mathcal{E}_3$.
The joint behaviour of $p$ and $\rho$ is in accordance with the theoretical bounds found in \cite{fontana2020model} and recalled in \eqref{corint} and  \eqref{cornint}.
In this case, where $d=3$, the minimal correlation is $-0.5$ and it is attained for $p=\frac{1}{3}$ and $p=\frac{2}{3}$.

\begin{figure}[h]
\caption{Bivariate  distribution of the first order moment and correlation  across $\mathcal{E}_3$}\centering
\includegraphics[width=0.4\textwidth]{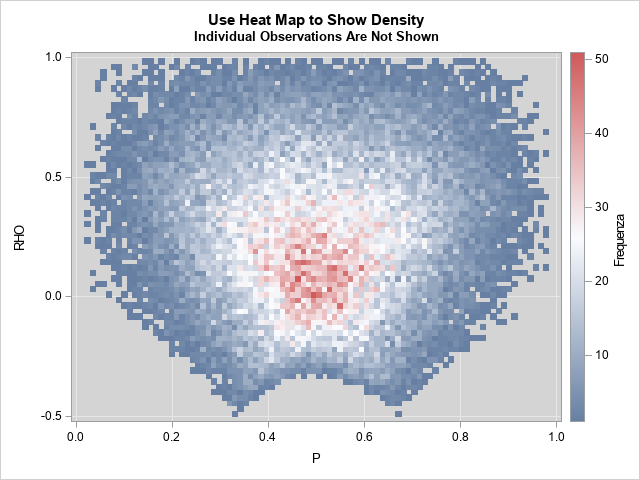}
\label{fig05}
\end{figure}

\subsubsection{Application 2}
We study:
\begin{enumerate}
\item the distribution of the second-order moment in  $\mathcal{E}_6(0.4)$, the class of exchangeable distributions of dimension $d=6$ and mean $p=0.4$;
\item the distribution of the entropy in the same class $\mathcal{E}_6(0.4)$;
%\item the joint distribution of the first-order moment and the entropy in the class of exchangeable function of dimension $d=6$, $\mathcal{E}_6$.
\end{enumerate}

The ray densities of $\mathcal{E}_6(0.4)$ are 12 and are given in Table \ref{tabray6}.

\begin{table}
\begin{center}\caption{Ray densities of  $\mathcal{E}_6(0.4)$}
\begin{tabular}{|r |r| r| r| r|r|r|r|r|r|r|r|r|} \label{tabray6}
$y$&$r_1(y)$ & 	$r_2(y)$ &	$r_3(y)$ & $r_4(y)$ &$r_5(y)$ & $r_6(y)$&$r_7(y)$& $r_8(y)$&$r_9(y)$&$r_{10}(y)$ & $r_{11}(y)$ & $r_{12}(y)$ \\
\hline
0&0.2 &	0.4 &	0.52 &	0.6 &	0 &	0 &	0 &	0 &	0 &	0 &	0 &	0 \\
1&0 &	0 &	0 &	0 &	0.3 &	0.533 &	0.65 &	0.72 &	0 &	0 &	0 &	0 \\
2&0 &	0 &	0 &	0 &	0 &	0 &	0 &	0 &	0.6 &	0.8 &	0.867 &	0.9 \\
3&0.8 &	0 &	0 &	0 &	0.7 &	0 &	0 &	0 &	0.4 &	0 &	0 &	0 \\
4&0 &	0.6 &	0 &	0 &	0 &	0.467 &	0 &	0 &	0 &	0.2 &	0 &	0 \\
5&0 &	0 &	0.48 &	0 &	0 &	0 &	0.35 &	0 &	0 &	0 &	0.133 &	0 \\
6&0 &	0 &	0 &	0.4 &	0 &	0 &	0 &	0.28 &	0 &	0 &	0 &	0.1 \\

\end{tabular}
\end{center}
\end{table}

The ray densities are points in $\RR^7$ which lie in a subspace of dimension $d-1=6-1=5$. Using standard Principal Component
Analysis we can project these 12  points to $\RR^5$.
We have 38 tetrahedra,   28 of which have almost zero volume.

We proceed as in the previous application to find the $2$-order moment distribution across the polytope. We first compute its distribution across each tetrahedron  using an exact and iterative formula and then we  mix the conditional cdfs as in \eqref{prob_tot}.
 Figure \ref{fig6-1}, left side,    shows the pdf of the mixture obtained from the exact numerical  cdf of $\mu_2$.

\begin{figure}[h]
\caption{Distribution of the $2$-order moment across $\mathcal{E}_6(0.4)$ (left side) and empirical distribution of the entropy across $\mathcal{E}_6(0.4)$ (right side)}\centering
\begin{subfigure}[b]{0.3\textwidth}
\includegraphics[width=\textwidth]{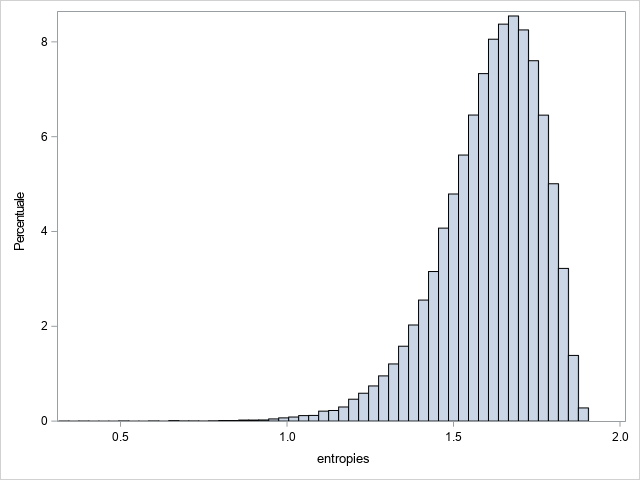}
\end{subfigure}
\begin{subfigure}[b]{0.3\textwidth}
\includegraphics[width=\textwidth]{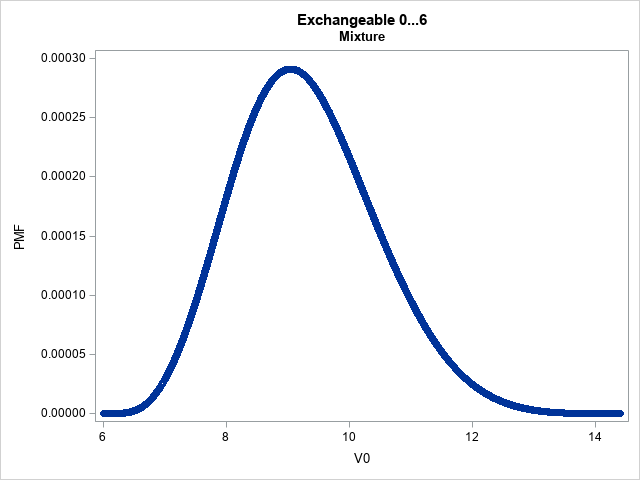}
\end{subfigure}
\label{fig6-1}
\end{figure}

The simulated pdf of the entropy in the  class $\mathcal{E}_6(0.4)$ is found using the methodology in Section \ref{sec:nonexp} and it is shown in the right side of  Figure \ref{fig6-1}.

%
%
%\begin{figure}[h]
%\caption{}\centering
%\includegraphics[width=0.3\textwidth]{entropy-hist_d_6.png}
%\label{fig03}
%\end{figure}
%
%We conclude with the joint editribution of the $2$-order moement and the correlation across $\mathcal{E}_6$.
%
%
%\begin{figure}[h]
%\caption{Bivariate  distribution of the first order moment and correlation  across $\mathcal{E}_3$}\centering
%\includegraphics[width=0.3\textwidth]{bivariato-p-e-rho.png}
%\label{fig04}
%\end{figure}
%

\section{Estimate and testing}\label{E-T}

\subsection{Maximum likelihood estimation}

We focus on the maximum likelihood estimation  in the classes of exchangeable distributions and exchangeable distributions with given margins, $\EEE$ and $\EEE(p)$ respectively.

\begin{proposition}

A  maximum likelihood estimator of $\ff\in \EEE$ ($\EEE(p)$) always exists.
\end{proposition}

\begin{proof}
$\EEE$ ($\EEE(p)$)  is a closed convex sets in $\RR^d$, hence it is compact and the likelihood functions for the  models in \eqref{L1} are continuous.
\end{proof}

The maximum likelihood estimator (MLE) in the class $\EEE$ can be found analytically using the map in \eqref{map0}.

Let us assume to observe a sample of size $n$ drawn from a $d$-dimensional Bernoulli distribution $\XX$ and let $Y=\sum_{i=1}^dX_i$ have pmf  $\pp_Y=(p_0, \ldots, p_d)$ that gives rise to counts  $\boldsymbol{N}$.
The count $\NNN=(N_0, \ldots, N_d)$
has a multinomial distribution with parameters $n, \pp$, i.e. $\NNN\sim \text{Multinomial}(n, \pp)$, where the parameter $\pp$ belongs to the  $d$-simplex  $\Delta_{d}$.
The likelihood function  is
\begin{equation}
L(\boldsymbol{n}; \pp)=P(N_i=n_i, i\in I)=\binom n {n_0\cdots n_d} \prod_{j=0}^d(p_j)^{n_j},
\end{equation}
where we set $0^0:=1$. The MLE  is the solution of the constrained maximization problem

 \begin{equation}
\begin{split}
&max_{\pp}\log L(\boldsymbol{n}; \pp), \\ &\,\,\,\,\,\,\,\,\,\,\, \text{sub} \\
&\sum_{i=0}^dp_j-1=0\\
\end{split}
\end{equation}
By using  the Lagrange multipliers we find:
$$\hat{p_j}_1=\frac{N_j}{n}.$$
The MLE in the class $\EEE$  is
$$\hat{f_j}=\frac{\hat{p_j}}{\binom d {j}}=\frac{\frac{N_j}{n}}{\binom d {j}},\,\,\, i=0,\ldots, d$$

We now consider the class $\EEE(p)$. The MLE estimator in $\mathcal{D}_d(dp)$ can be numerically found  by  solving the constrained maximization problem:

 \begin{equation}
\begin{split}
&max_{\pp}\log L(\boldsymbol{n}; \pp), \\ &\,\,\,\,\,\,\,\,\,\,\, \text{sub} \\
&\sum_{i=0}^dp_j-1=0\\
&\sum_{i=0}^djp_j-pd=0,\\
\end{split}
\end{equation}
Therefore to have the MLE in $\EEE(p)$ we use the map $H$ again.

We can also use a direct approach and look for the MLE  in $\EEE(p)$.
Let us assume to observe a sample of size $n$ drawn from a $d$-dimensional Bernoulli distribution $\XX$ with pmf  $(f(\boldsymbol{x}):\boldsymbol{x}\in\mathcal{X}_d)$ that gives rise to counts  $\NNN=(N_1,\ldots N_m)$, where
 $m=2^d$.  Then the count $\NNN$ has a multinomial distribution with parameters $n, \ff^{\chi}$, i.e. $\NNN\sim \text{Multinomial}(n, \ff^{\chi})$, where the parameter $\fff=(f^{\chi}_i: i=1,\ldots, m):=(f(\boldsymbol{x}):\boldsymbol{x}\in\mathcal{X}_d)$ belongs to the  $m-1$-simplex  $\Delta_{m-1}$.
The likelihood function  is
\begin{equation}
L(\boldsymbol{n}; \fff)=P(N_i=n_i, i\in I)=\binom n {n_1\cdots n_m} \prod_{j=1}^m(f^{\chi}_j)^{n_j},
\end{equation}
where we set $0^0:=1$.
%The maximum likelihood estimator (MLE) in the simplex $\mathcal{B}_d$ of $d$-dimensional Bernoulli distributions can be easily found and it is $\hat{\fff}=(\frac{N_1}{n},\ldots,\frac{
%N_m}{n})$.
%
%We focus on the MLE in the classes of exchangeable distributions and exchangeable distributions with given margins, $\EEE$ and $\EEE(p)$ respectively. Let $\EEE^*$ be one of the classes $\EEE$ or $\EEE(p)$.
%

If we assume that $\XX\in\EEE(p)$, then $\fff=\fff(\lambda)$ has the form:
\begin{equation}
\fff=\sum_{i=1}^{k}\lambda_i\ee^{\chi}_i,
\end{equation}
where $\ee^{\chi}_i=(e(\boldsymbol{x}):\boldsymbol{x}\in\mathcal{X}_d)$ are the extremal points of $\EEE(p)$ and $n_p$ is their number. The likelihood function of the count $\NNN$  is
\begin{equation}\label{L1}
L(\boldsymbol{n}; \ff^{\chi})=\binom n {n_1\cdots n_m} \prod_{j=1}^m(\sum_{i=1}^{n_p}\lambda_ie^{\chi}_{ij})^{n_j}
\end{equation}
 where $\lambda=(\lambda_1,\ldots, \lambda_{n_p})\in \Delta_{n_p}$. The MLE  can be found by maximizing the log-likelihood function in the simplex.
%
%
%\begin{remark}
%Notice that when the dimension increases we can greatly reduce the computational effort by using the one-to-one correspondence between $\mathcal{E}_d(p)$ and $\mathcal{S}_d(p)$ and look for a MLE estimator in $\mathcal{S}_d(p)$. Thank to
% the map $H$ in
%\eqref{map0} which provides a one to one relationship between the ray densities of the two classes, the MLE estimates $\hat{\lambda}_i, i=1,\ldots, k$ are the same in the two spaces. Actually,  we can find the MLE stimator in
%$\mathcal{S}_d(p)$  and, through the map $H$, the MLE estimator in $\mathcal{E}_d(p)$.
%
%
%\end{remark}

\begin{example}\label{ex12}
Let $\XX\in \mathcal{E}_2(1/2)$. We have two ray densities: the upper and lower Fr\'echet  bound (\cite{fontana2018representation}) $\ee_U$ and $\ee_L$. The count $\NNN$ has support on four points and the likelihood becomes:

\begin{equation}
L(\boldsymbol{n}; \ff^{\chi})=\binom n {n_1\cdots n_4} \prod_{j=1}^4(\lambda_1e_{Uj}+\lambda_2e_{Lj})^{n_j},  \lambda\in \Delta_2.
\end{equation}
By standard computations we find

\begin{equation}
L(\boldsymbol{n}; \ff^{\chi})=\binom n {n_1\cdots n_4} (\frac{\lambda_1}{2})^{n_1+n_4} (\frac{\lambda_2}{2})^{n_2+n_3},  \lambda\in \Delta_2.
\end{equation}

The MLE  can be found using the log-likelihood and the Lagrange multipliers and it is:

$$\hat{\lambda}_1=\frac{N_1+N_4}{n},\,\,\,\,\,\, \hat{\lambda}_2=\frac{N_2+N_3}{n}.$$
The same result has been found in \cite{marchetti2016palindromic} for the palindromic  Bernoulli distributions, that in the 2-dimensional case coincide with  the whole Fr\'echet class $\EEE(1/2)$.
\end{example}
%
%\begin{remark}
%Another way to find MLE estimator is to solve the constrained maximization problem
%
%\begin{equation}
%\max_{\ff^{\chi}}\binom n {n_1\cdots n_m} \prod_{j=1}^m(f^{\chi}_j)^{n_j},
%\end{equation}
%under the linear constraints on $\fff$ to impose exchangeability and moments.
%
%\end{remark}
\subsection{Testing}\label{Test}
Let $\EEE^*$ be one of the classes $\EEE$ or $\EEE(p)$.
This section provides a generalized likelihood ratio (GLR) test for $$H_0: \fff\in \EEE^*$$ versus $$H_A: \fff\in \mathcal{B}_d\setminus\EEE^*,$$
where in this case the class $\EEE^*$ is a $d$-simplex o a $d$-polytope and $\mathcal{B}_d$ is a $2^d-1$-simplex.

Let  $\hat{\fff}=(\frac{N_1}{n},\ldots,\frac{N_m}{n})$ be the MLE estimator for
$\NNN\sim \text{Multinomial}(\boldsymbol{n}, \fff)$, where the parameter $\fff=(f^{\chi}_i: i=1,\ldots, m):=(f(\boldsymbol{x}):\boldsymbol{x}\in\mathcal{X}_d)$ belongs to the  $m-1$-simplex  $\Delta_{m-1}$
and let $\hat{\lambda}$ be the MLE estimator for  $\XX\in\EEE^*$ with pmf $\fff=\fff(\lambda)$ as determined in the previous section. The GLR statistics is

\begin{equation}\label{statTest}
\Lambda(\NNN)=\frac{\prod_{i=1}^m(f^{\chi})_j(\hat{\lambda})^{N_j}}{\prod_{i=1}^m(\frac{N_j}{n})^{N_j}},
\end{equation}
where $\NNN$ is the count arised from $\XX$.
The $\alpha$-level critical region is defined by
\begin{equation}
\alpha=P_0(\Lambda(\NNN)<c),
\end{equation}
where $P_0$ is the probability measure under $H_0$. The value  $c$ is obtained observing that $-2\log(\Lambda(\NNN))$ is approximatively a $\chi^2_k$ distribution with $k=m-1-\dim(\mathcal{E}^*_d)$, where $\dim(\mathcal{E}^*_d)=d$ if $\mathcal{E}^*_d=\mathcal{E}_d$ and $\dim(\mathcal{E}^*)=d-1$ if $\mathcal{E}^*_d=\mathcal{E}_d(p)$.
%The class $\EEE$ is defined by the exchangeability conditions that are $2^d-d-1$ plus the usual  condition for $\ff^{\chi}$ being a legitimate pmf.
%The degree of fredoom for $\EEE$ is $k=2^d-(2^d-d)=d$ and  for $\EEE(p)$ is $k=2^d-(2^d-d+1)=d-1$, i.e $k=2^2-1-1=2$ for $\mathcal{E}_2$ and  $k=2^2-1-2=1$ for $\mathcal{E}_2(p)$. %See \cite{} for  further details on multinomial models tests.
\begin{example}
Consider the MLE  of $\XX\in \mathcal{E}_2(1/2)$ in Example \ref{ex12}. The test statistics for   $$H_0: \fff\in \mathcal{E}_2(1/2)$$ versus $$H_A: \fff\in \mathcal{B}_2\setminus\mathcal{E}_2(1/2),$$
is

\begin{equation}
\Lambda(\NNN)=\frac{\big(\frac{N_1+N_4}{2n}\big)^{N_1+N_4}\big(\frac{N_2+N_3}{2n}\big)^{Y_2+Y_3}}{\prod_{i=1}^4\frac{N_j}{n}^{N_j}},
\end{equation}
and $-2\log(\Lambda(\NNN))$ has  approximatively a $\chi^2_2$ distribution. If we consider $\alpha=0.10$ and $c_1=5.991$ the upper $0.95$ quantile of a $\chi_1^2$ distribution the critical region is defined by $c=e^{-2c_1}$.
\end{example}

\subsection{Application}

Over the Spring 2009 semester, two Berkeley undergraduates  undertook  40,000 tosses of a coin. The dataset and the description of the protocol which followed are available at \url{https://www.stat.berkeley.edu/~aldous/Real-World/coin_tosses.html}.
Here, we rearrange this dataset as if the tosses had been undertaken five at a time and we use this dataset to find  the  MLE  in the class $\mathcal{E}_5(\frac{1}{2})$.  After finding the MLE  in $\mathcal{E}_5(\frac{1}{2})$, we perform the GLR test described in Section \ref{Test}.

To simplify the computations  we look for the ML estimates in  $\mathcal{S}_5(\frac{1}{2})$. We have nine ray densities, provided in Table \ref{tabMLE}. The MLE estimate
$\hat{\boldsymbol{\lambda}}=(\lambda_i, \, i=1,\ldots, 9)$  which is also the ML estimate in  $\mathcal{E}_5(\frac{1}{2})$  is
\begin{equation}\label{lambda}\hat{\boldsymbol{\lambda}}=(0.1,	0.019,0.015,0.188,0.202,0.008,0.173,0.174,0.121).\end{equation}  For completeness we also exhibit the estimated ML pmf $p_{MLE}$ in  $\mathcal{S}_5(\frac{1}{2})$ in the last column of Table \ref{tabMLE}.
%
%\begin{table}
%\begin{footnotesize}
%\begin{center}\caption{Ray density of  $\mathcal{S}_5(\frac{1}{2} )\equiv \mathcal{E}_5(\frac{1}{2})$ and estimated MLE pmf}
%\begin{tabular}{|r |r| r| r| r|r|r|r|r|r|r| } \label{tabMLEest}				
%0.1&0.019,0.015,0.188,0.202,0.008,0.173,0.174,0.121)$
\begin{table}
\begin{footnotesize}
\begin{center}\caption{Ray density of  $\mathcal{S}_5(\frac{1}{2} )\equiv \mathcal{E}_5(\frac{1}{2})$ and estimated MLE pmf}
\begin{tabular}{|r |r| r| r| r|r|r|r|r|r|r| } \label{tabMLE}
$y$&$r_1(y)$ & 	$r_2(y)$ &	$r_3(y)$ & $r_4(y)$ &$r_5(y)$ & $r_6(y)$&$r_7(y)$& $r_8(y)$&$r_9(y)$ & $p_{MLE}(y)$\\
\hline
0&0.167 &	0.375 &	0.5 &	0 &	0 &	0 &	0 &	0 &	0 &0.031\\
1&0 &	0 &	0 &	0.25 &	0.5 &	0.625 &	0 &	0 &	0 &0.153\\
2&0 &	0 &	0 &	0 &	0 &	0 &	0.5 &	0.75 &	0.833& 0.318\\
3&0.833 &	0 &	0 &	0.75 &	0 &	0 &	0.5 &	0 &	0 &0.311\\
4&0 &	0.625 &	0 &	0 &	0.5 &	0 &	0 &	0.25 &	0&0.156 \\
5&0 &	0 &	0.5 &	0 &	0 &	0.375 &	0 &	0 &	0.167& 0.031\\
\end{tabular}
\end{center}
\end{footnotesize}
\end{table}

We now perform  the GLR test for $$H_0: \fff\in \mathcal{E}_5(\frac{1}{2})$$ versus $$H_A: \fff\in \mathcal{B}_d\setminus\mathcal{E}_5(\frac{1}{2}).$$

Let  $\hat{\fff}=(\frac{N_1}{40000},\ldots,\frac{N_{32}}{40000})$ be the MLE  for
$\NNN\sim \text{Multinomial}(40000, \fff)$
and let $\hat{\boldsymbol{\lambda}}$ be the MLE estimator for  $\XX\in\mathcal{E}_5(\frac{1}{2})$ provided in \eqref{lambda}. The GLR statistics is $\Lambda(\NNN)$  in \eqref{statTest}. Since $-2\log(\Lambda(\NNN))$ is approximatively a $\chi^2_k$ distribution with $k=27$ degree of freedom,  its observed value is $39.49$ and $\chi^2_{0.95}=40.113$ we do not reject the null hypotesis at level $0.05$.

\section{Further developements}\label{P-exch}
This section shows that the geometrical structure of exchangeable Bernoulli pmf holds in a more general framework, partial exchangeability. We also show that as well as exchangeable pmf are in a one to one relationship with discrete distributions, partially exchangeable pmf are in a one to one  relationship with multivariate discrete distributions.  The results we present here open the way to the study of this more general class of multivariate Bernoulli pmf.

\begin{definition}
Let $\mathcal{G}$ be a partition of $I=\{1,\ldots, d\}$. A multivariate Bernoulli distribution $f(\xx)$ is partially exchangeable if $f(\sigma(\xx))=f(\xx)$ for any $\sigma\in \mathcal{P}_d$ such that $\sigma(G)=G$ for any $G\in \mathcal{G}$. We say that $\sigma$ and $f(\xx)$ are compatible with $\mathcal{G}$. We denote by $\mathcal{P}_d(\mathcal{G})$ the set of partitions compatible with $\GGG$ and $\EEE(\mathcal{G})$ the family of partially exchangeable distributions compatible with $\mathcal{G}$.
\end{definition}

Partial exchangeability is an extension of exchangeability, that is recovered by choosing the trivial partition $\GGG=\{I\}$.

Let   $\mathcal{D}_{d_1, \ldots, d_n}$ be the class of multivariate discrete distributions with support on $J_1\times\ldots\times J_n$ and  $J_k=\{0,\ldots, d_k\}$  and  $\mathcal{D}_{d_1, \ldots, d_n}(\mmu)=\mathcal{D}_{d_1, \ldots, d_n}(\mu_1,\ldots,\mu_n)$ the class of multivariate discrete distributions with support on $J_1\times\ldots\times J_n$ and mean vector $\mmu=(\mu_1, \ldots, \mu_n)$.

\begin{theorem}\label{Thm}
Let  $\GGG=\{G_1, \ldots, G_n\}$, and $d_j=\#G_j$.
There  is a one to one map $F_{\GGG}$ between $\EEE(\GGG)$ and $\mathcal{D}_{d_1, \ldots, d_n}$.

\end{theorem}
\begin{proof}
Let $f\in \EEE(\mathcal{G})$. Since $f(\boldsymbol{x}%
)=f(\sigma (\boldsymbol{x}))$ for any $\sigma \in \mathcal{P}_d(\GGG)$, any
mass function $f(\xx)$ in $\mathcal{E}_d(\GGG)$ uniquely defines a function $g: J_1\times\ldots\times J_n\rightarrow$, where $J_k=\{0,\ldots, d_k\}$ and $d_k=\#G_k$ given by $
g(j_1, \ldots, j_n):=f(\boldsymbol{x})$ if $\boldsymbol{x}=(x_{1},\ldots ,x_{d})\in
\mathcal{X}_d$ and $\#\{x_h\in G_i:x_h=1\}=j_i, \,\,\, i=1,\ldots n$.
The map
\begin{equation}  \label{map2}
\begin{split}
F_{\GGG}: \mathcal{E}_d(\GGG)&\rightarrow \mathcal{D}_{d_1, \ldots, d_n} \\
f &\rightarrow p_D.
\end{split}%
\end{equation}
where $p_D(j_1, \ldots, j_n)={\binom{d_1}{j_1}}\ldots{\binom{d_n}{j_n}}g(j_1,\ldots, j_n)$ is a bijection.

\end{proof}
Notice that if $\XX$ is partially exchangeable, each $d_j$-dimensional margin of the form $(X_{i})_{i\in G_j}$ is a vector of exchangeable Bernoulli variables.

\begin{remark}
Let $\mathcal{S}(\GGG)$  be the class of random variables $\YY=(Y_1, \ldots, Y_n)$ defined by:
\begin{equation}
Y_j=\sum_{h\in G_j}X_h,
\end{equation}
then $p_{\YY}(j_1,\ldots, j_n)=p_D(j_1, \ldots, j_n)$. Thus,  $\mathcal{S}(\GGG)=\mathcal{D}_{d_1, \ldots, d_n}.$

\end{remark}

\begin{corollary}
The class $\EEE(\GGG)$ is a $d_{\GGG}$-simplex, where  $d_{\GGG}=(d_1+1)\times\ldots\times (d_n+1)-1$.
The class $\EEE(\GGG)(\mmu)$ of partially exchangeable distributions compatible with $\GGG$ and set of moments $\mmu$  is a $d_{\GGG}$-polytope, where  $d_{\GGG}=(d_1+1)\times\ldots\times (d_n+1)-1$.
\end{corollary}

This last result implies that all the analysis performed in the previous sections can be extended to partially exchangeable distributions.

\begin{example}\label{ex-pair}
Let $\XX\in \mathcal{E}_4(\GGG)$, where $\GGG=\{\{1,2\},\{3,4\}\}$. Let  $\YY=(Y_1, Y_2)$ defined by
\begin{equation}
Y_1=X_1+X_2, \,\,\, Y_2=X_3+X_4.
\end{equation}
$\YY\in \mathcal{D}_{2, 2}$ and $p_S(j_1, j_2)={\binom{2}{j_1}}{\binom{2}{j_2}}f(j_1, j_2)$, $(j_1, j_2)\in J_1\times J_2$. Therefore the vector  $\pp_Y=(p_Y(j_1, j_2))_{j_1, j_2\in J_1\times J_2}$   is a point in $\RR^9$ and $ \mathcal{E}_4(\GGG)$ is a 8-simplex in $\RR^9$. %The extremal rays of  $\mathcal{S}(\GGG)$ (and of  $\mathcal{D}_{2, 2}$) are $\gg_j=(0,\ldots, 1, \ldots,0)\in \RR^9, j=0,\ldots 8$.

\end{example}

\subsection{Given means: the class $\EEE(\GGG)(\pp)$}
Let $\XX\in \EEE(\GGG)$, $\GGG=\{G_1,\ldots, G_n\}$ and assume that  $E[X_i]=p_j$ if $i\in G_j$. Let $\pp=(p_1,\ldots, p_n)$ the mean vector. We consider here the class $ \EEE(\GGG)(\pp)$ of Bernoulli pmf with mean vector $\pp$. The map $F_{\GGG}$ induces a bijection between $\EEE(\GGG)(\pp)$ and $\mathcal{D}_{d_1,\ldots, d_n}(\dd\pp)=\mathcal{S}(\GGG)(\dd\pp)$, where $\dd\pp:=(d_1p_1,\ldots, d_np_n)$.

\begin{proposition}
\label{sd} Let $\boldsymbol{Y}\in  \mathcal{D}_{d_1,\ldots, d_n}$  and let $p_{\YY}$ be its pmf. Then
\begin{equation*}
\YY\in \mathcal{S}_d(\dd\pp)\,\,\, \Longleftrightarrow\,\,\,
\sum_{j_1=0}^{d_l}\cdots\sum_{j_n=0}^{d_n}(j_k-p_kd_k)p_Y(j_1,\ldots, j_n)=0,\,\,\, k=1,\ldots, n.
\end{equation*}
\end{proposition}

\begin{proof}
Let $\YY\in \mathcal{D}_{d_1,\ldots, d_n} $.  By Theorem \ref{Thm} $\YY\in \SSS(\GGG)(\dd\pp)$ iff $E[Y_k]=p_kd_k$.
%\begin{equation*}
%Y\in \SSS(p) \Longleftrightarrow E[Y]=pd.
%\end{equation*}
It holds
\begin{equation*}
E[Y_k]=p_kd_k \Longleftrightarrow E[Y_k-p_kd_k]=0 \Longleftrightarrow  \sum_{j_1=0}^{d_l}\cdots\sum_{j_n=0}^{d_n}(j_k-p_kd_k)p_Y(j_1,\ldots, j_n)=0.
\end{equation*}
\end{proof}
From this proposition and Proposition \ref{multinulli} it follows:
\begin{corollary}
The extremal points of the polytope  $\EEE(\GGG)(\pp)$ have support on at most $n+1$  points.
\end{corollary}

\begin{example}

Let $\XX\in \mathcal{E}_4(\GGG)$, where $\GGG=\{\{1,2\},\{3,4\}\}$ as in Example \ref{ex-pair} and let $\pp=(p_1, p_2)$ the mean vector. The convex polytope  $\EEE(\GGG)(\pp)$ is the set of solutions of the linear system:
\begin{equation*}
\left\{
\begin{array}{cc}
-2p_1(p_{00}+p_{01}+p_{02})+(1-2p_1)(p_{10}+p_{11}+p_{12})+(2-2p_1)(p_{20}+p_{21}+p_{22})=0\\
-2p_2(p_{00}+p_{10}+p_{20})+(1-2p_2)(p_{01}+p_{11}+p_{21})+(2-2p_2)(p_{02}+p_{12}+p_{22})=0
\end{array}
\right.,
\end{equation*}
therefore the extremal rays have support on at most three points. %If we consider tha case $p_1=p_2=\frac{1}{2}$ the extremal rays are given in Table...
As an example Table \ref{tabraype} provides the extremal rays for $\pp=(\frac{1}{2}, \frac{1}{4})$.

\begin{table}
\begin{footnotesize}
\begin{center}\caption{Ray density of  $\mathcal{S}(\GGG)(1, \frac{1}{2} )\equiv \EEE(\GGG)(\frac{1}{2}, \frac{1}{4})$}
\begin{tabular}{|r |r| r| r| r|r|r|r|r|r|r|r|r|r|r| } \label{tabraype}
$y$&$r_1(y)$ & 	$r_2(y)$ &	$r_3(y)$ & $r_4(y)$ &$r_5(y)$ & $r_6(y)$&$r_7(y)$& $r_8(y)$&$r_9(y)$&$r_{10}(y)$ & $r_{11}(y)$ & $r_{12}(y)$ & $r_{13}(y)$ & $r_{14}(y)$\\
\hline
00&0 &	0 &	0 &	0 &	0 &	0 &	0 &	0 &	0.5 &	0.25 &	0.25 &	0.25 &	0.5 &	0.375 \\
10&0 &	0.5 &	0.5 &	0.5 &	0.75 &	0.667 &	0.667 &	0.75 &	0 &	0 &	0 &	0.5 &	0 &	0 \\
20&0.5 &	0 &	0 &	0.25 &	0 &	0 &	0 &	0 &	0 &	0.25 &	0.5 &	0 &	0.25 &	0.375 \\
01&0.5 &	0 &	0.25 &	0 &	0 &	0 &	0.167 &	0 &	0 &	0 &	0 &	0 &	0 &	0 \\
11&0 &	0.5 &	0 &	0 &	0 &	0 &	0 &	0 &	0 &	0.5 &	0 &	0 &	0 &	0 \\
21&0 &	0 &	0.25 &	0 &	0 &	0.167 &	0 &	0 &	0.5 &	0 &	0 &	0 &	0 &	0 \\
02&0 &	0 &	0 &	0.25 &	0 &	0.167 &	0 &	0.125 &	0 &	0 &	0.25 &	0 &	0 &	0 \\
12&0 &	0 &	0 &	0 &	0.25 &	0 &	0 &	0 &	0 &	0 &	0 &	0 &	0 &	0.25 \\
22&0 &	0 &	0 &	0 &	0 &	0 &	0.167 &	0.125 &	0 &	0 &	0 &	0.25 &	0.25 &	0 \\
\end{tabular}
\end{center}
\end{footnotesize}
\end{table}

\end{example}
%\section{Conclusions\label{fine}}

\section*{Acknowledgements}
The authors gratefully acknowledge financial support from the Italian Ministery of Education, University and Research, MIUR, "Dipartimenti di Eccellenza"
grant 2018-2022.

\bibliographystyle{ieeetr}
\bibliography{biblioRF}

\end{document}